\documentclass[10pt, reqno]{amsart}
\usepackage{lipsum}
\usepackage{a4wide}
\usepackage{amssymb}
\usepackage{amsmath}
\usepackage{amsthm}
\usepackage{tikz}
\usepackage{amscd}
\usepackage{hyperref}
\usepackage[all]{xy}
\usepackage{enumitem}
\hypersetup{colorlinks,linkcolor={blue},citecolor={blue},urlcolor={red}}
\usepackage{color}
\theoremstyle{plain}
\usepackage[margin=2.4cm]{geometry}
\numberwithin{equation}{section}

\newtheorem{theorem}{Theorem}[section]

\newtheorem{lemma}[theorem]{Lemma}
\newtheorem{remark}[theorem]{Remark}
\newtheorem{proposition}[theorem]{Proposition}

\title[Tate cohomology of generic representations of ${\rm GL}_3$]{Tate cohomology and local base change of generic representations of ${\rm GL}_3$ -- non-banal case}
\author{Sabyasachi Dhar}
\begin{document}
	\maketitle
\begin{abstract}
Let $F$ be a finite extension of $\mathbb{Q}_p$, and let $E$ be a finite Galois extension of $F$ with degree of extension $l$, where $l$ and $p$ are distinct odd primes. Let $\pi_F$ be an integral, $l$-adic generic representation of ${\rm GL}_3(F)$, and let $\pi_E$ be the base change lifting of $\pi_F$ to ${\rm GL}_3(E)$. Let $J_l(\pi_F)$ (resp. $J_l(\pi_E)$) be the unique generic sub--quotient of the mod-$l$ reduction of $\pi_F$ (resp. $\pi_E$). In this article, using the local converse theorem over local Artinian $\overline{\mathbb{F}}_l$-algebras, we prove that the Frobenius twist of $J_l(\pi_F)$ is isomorphic to the Tate cohomology group $\widehat{H}^0({\rm Gal}(E/F),J_l(\pi_E))$.  The result of this article removes the hypothesis that the prime $l$ does not divide the pro-order of ${\rm GL}_2(F)$ in the previous work \cite{nadimpalli2024tate}. 
\end{abstract}

\section{Introduction}
Let $F$ be a finite extension of $\mathbb{Q}_p$, and let 
$E$ be a Galois extension of $F$ with $[E:F]$ a prime 
number $l$ distinct from $p$. The theory of base change
lifting provides a map 
$${\rm BC}_{E/F}:{\rm Irr}_{\overline{\mathbb{Q}}_l}(F)\rightarrow {\rm Irr}_{\overline{\mathbb{Q}}_l}(E),$$
where ${\rm Irr}_{\overline{\mathbb{Q}}_l}(K)$ denotes the set of irreducible
smooth ${\overline{\mathbb{Q}}_l}$-representations of ${\rm GL}_n(K)$. Here, $K$ denotes a non-Archimedean local field with residue
characteristic $p$. For 
$\pi\in {\rm Irr}_{\overline{\mathbb{Q}}_l}(F)$, the 
representation $\Pi={\rm BC}_{E/F}(\pi)$ is stable under the 
action of ${\rm Gal}(E/F)$. Moreover, the representation
$\pi$ is integral if and only if $\Pi$ is integral (see Lemma \ref{bc_integral}). In this 
article, we explicitly compute the Tate cohomology 
group $\widehat{H}^i({\rm Gal}(E/F), r_l(\Pi))$,
where $r_l(\Pi)$ is the mod-$l$ reduction of $\Pi$--whenever
it is integral. The 
primary novelty of this article is that we consider all 
values of $l$, and we use methods of local converse theorem over Artin local $\overline{\mathbb{F}}_l$-algebras. 
In particular, the present work includes the case where
$l=n=3$; thus allowing the situation where the base change lift $\Pi$ is non-cuspidal. 

In their foundational work \cite[Conjecture 6.3]{Venkatesh-Treumann}, Treumann--Venkatesh  
conjecture that one of Jordan--Holder factors 
of the $\overline{\mathbb{F}}_l$-representation  
$\widehat{H}^i({\rm Gal}(E/F), r_l(\Pi))$ of 
${\rm GL}_n(F)$ is the 
mod-$l$ reduction of $\pi$ twisted by the Frobenius 
automorphism of the field of coefficients $\overline{\mathbb{F}}_l$. 
Their conjectures predict that the group 
$\widehat{H}^i({\rm Gal}(E/F), r_l(\Pi))$ is 
cuspidal whenever $\pi$ is cuspidal. Very interestingly, 
in the case where $l=n=3$, and when $\Pi$ is non-cuspidal,
and thus a principal series representation, the group
$\widehat{H}^i({\rm Gal}(E/F), r_l(\Pi))$ are 
conjectured to be cuspidal. Thus producing a mod-$l$ cuspidal representation starting with a non-cuspidal representation. Our main theorem (see Theorem \ref{intro_3_thm} for the precise statement) proves Treumann--Venkatesh conjecture for all generic $\overline{\mathbb{Q}}_l$-representations of ${\rm GL}_3(E)$, obtained as base change lifting of integral, generic $\overline{\mathbb{Q}}_l$-representations of ${\rm GL}_3(F)$. As a consequence of this result, we get that the Tate cohomology group $\widehat{H}^i({\rm Gal}(E/F), r_l(\Pi))$ is isomorphic to the cuspidal representation  $r_l(\pi)\otimes_{\rm Frob}\overline{\mathbb{F}}_l$ (Theorem \ref{bc_cuspidal}).

We now introduce some notations to state the main result. These are very standard throughout and are chosen to be consistent with \cite{nadimpalli2024tate}, to which we shall often refer. 
Let $F$ be a finite extension of $\mathbb{Q}_p$, and let $E$ be a finite Galois extension of $F$ with degree of extension $l$, where $l$ and $p$ are distinct primes. Let $\Gamma$ be the Galois group of the extension $E/F$. Let $(\pi_F,V)$ be a smooth, integral, $l$-adic generic representation of ${\rm GL}_n(F)$. Let $\mathcal{L}$ be a ${\rm GL}_n(F)$ invariant lattice in $V$. The semi--simplification of the mod-$l$ representation $\mathcal{L}\otimes_{\overline{\mathbb{Z}}_l}\overline{\mathbb{F}}_l$ is independent of the choice of $\mathcal{L}$ (see \cite[III, 5.11.a and 5.11.b]{Vigneras-l-modular}), and it is called the mod-$l$ reduction of $\pi_F$, denoted by $r_l(\pi_F)$. In general, $r_l(\pi_F)$ is not generic but it has a unique generic sub--quotient, denoted by $J_l(\pi_F)$.
Let $\pi_E$ be the base change lift of $\pi_F$ to ${\rm GL}_n(E)$ (see Subsection \ref{local_base_change} for the definition). Then $\pi_E$ is also integral, generic, and there is an isomorphsim $T:\pi_E\xrightarrow{\sim}\pi_E^\gamma$, for all $\gamma \in \Gamma$. The generic component $J_l(\pi_E)$ of the mod-$l$ reduction of $\pi_E$ is also stable under the action of $\Gamma$--induced by $T$, and we have the Tate cohomology $\widehat{H}^i(\Gamma, J_l(\Pi_E))$, which is a mod-$l$ representation of ${\rm GL}_n(F)$. The main theorem of this paper is the following.
\begin{theorem}\label{intro_3_thm}
Let $F$ be a finite extension of $\mathbb{Q}_p$, and let $E$ be a finite Galois extension of $F$ with $[E:F]=l$, where $l$ and $p$ are distinct odd primes. Let $\pi_F$ be a smooth, integral, generic, $\overline{\mathbb{Q}}_l$-representation of ${\rm GL}_3(F)$, and let $\pi_E$ be the base change lifting of $\pi_F$ to ${\rm GL}_3(E)$. Then the Frobenius twist $J_l(\pi_F)^{(l)}:=J_l(\pi_F)\otimes_{\rm Frob} \overline{\mathbb{F}}_l$ is the unique generic sub--quotient of the Tate cohomology group $\widehat{H}^0(\Gamma,J_l(\pi_E))$.
\end{theorem}
As discussed earlier, when $l=3$ and $\pi_F$ is cuspidal , then the base change lift $\pi_E$ is either a cuspidal representation or a principal series representation. In the previous work \cite[Theorem 6.7]{nadimpalli2024tate}, the second case was not dealt because the hypothesis that $l$ does not divide the pro-order of ${\rm GL}_2(F)$ implies $l\ne 3$, in which case $\pi_E$ is always cuspidal. So removing the hypothesis on $l$ in Theorem \ref{intro_3_thm} allows the case where the base change lift $\pi_E$ is non-cuspidal. In both cases, the Tate cohomology group $\widehat{H}^0(\Gamma,r_l(\pi_E))$ is cuspidal (Lemma \ref{Tate_irr}), and is infact isomorphic to the mod-$l$ cuspidal representation $r_l(\pi_F)^{(l)}$, by Theorem \ref{intro_3_thm}. This reveals a striking feature  of Treumann--Venkatesh's conjecture that the Tate cohomology of certain $\Gamma$ stable non-cuspidal mod-$l$ representations of ${\rm GL}_3(E)$ produces mod-$l$ cuspidal representations of ${\rm GL}_3(F)$.

T. Feng studied Treumann--Venkatesh's conjecture for general class of reductive algebraic groups (see \cite{feng2024smith}, \cite{feng2023modular}) defined over a local field using recent advances due to Genestier--Lafforgue and Fargues--Scholze. In the previous work \cite{nadimpalli2024tate} with S. Nadimpalli, the above theorem was proved for the base change lifting of integral $l$-adic generic representations of ${\rm GL}_n(F)$, under the hypothesis that  $l$ does not divide $|{\rm GL}_{n-1}(k_F)|$, for $n\geq 3$, where $k_F$ is the finite residue field of $F$. The condition on $l$ is required for certain vanishing result over $\overline{\mathbb{F}}_l$--which, in general, fails for arbitrary prime $l$ (see \cite[Section 4]{moss_mod_l_nilpotent_gamma}). In this article, exploiting the theory of smooth representations over Artin local $\overline{\mathbb{F}}_l$-algebras, we prove Theorem \ref{intro_3_thm}, where we remove the hypothesis that $l$ does not divide $\lvert{\rm GL}_2(k_F)\rvert$. We use the local converse theorem over Artin local $\overline{\mathbb{F}}_l$-algebras (see \cite[Theorem 1.2]{moss_mod_l_nilpotent_gamma} for the precise statement). The following is the essential step: Let $\psi:F\rightarrow W(\overline{\mathbb{F}}_l)^\times$ be a non-trivial additive character, where $W(\overline{\mathbb{F}}_l)$ is the ring of Witt vectors of $\overline{\mathbb{F}}_l$. Using \cite[Proposition 6.3]{nadimpalli2024tate}, we get that the Tate cohomology group $\widehat{H}^0(\Gamma,J_l(\pi_E))$ has a unique generic sub--quotient, say $\Xi$. Then, we prove the following identity of local $\gamma$-factors: 
$$
\gamma(X,\Xi,\eta,\psi) = 
\gamma(X,J_l(\pi_F)^{(l)},\eta,\psi),
$$ 
where $\eta$ varies over the set of all co-Whittaker $R[F^\times]$ modules and $R$ varies over the set of all Artin local $\overline{\mathbb{F}}_l$-algebras. In order to obtain the above equality, we use the machineries of Rankin-Selberg theory over Noetherian $W(\overline{\mathbb{F}}_l)$-algebras, provided by the work of \cite{Moss_Gammafactors}. The proof also uses properties of Deligne--Langlands $\gamma$-factors over Noetherian $W(\overline{\mathbb{F}}_l)$-algebras--which are made available by the seminal work of D. Helm (\cite[Theorem 1.1]{helm2015deligne}). 

The paper is organized as follows: In Section $2$, we recall some basic notions of smooth representations of $p$-adic groups. In Section $3$, we review Rankin--Selberg theory over Noetherian $W(\overline{\mathbb{F}}_l)$-algebras. In Section $4$, we discuss the Weil--Deligne representations. In Section $5$, we recall the local Langlands correspondence and the notion of local base change lifting. In Section $6$, we set up some initial results on Tate cohomology. In Section $7$, we prove Theorem \ref{intro_3_thm}. 
	
\section{Preliminaries}
In this part, we recall some standard notions in the representation theory over arbitrary Noetherian rings. For a reference, see \cite{Moss_Gammafactors}, \cite{Helm_Co-whittaker}.

\subsection{}
Let $l$ and $p$ be two distinct prime numbers. For any finite extension $K$ of $\mathbb{Q}_p$, we denote by $\mathfrak{o}_K$ the ring of integers of $K$ and the maximal ideal of $\mathfrak{o}_K$ is denoted by $\mathfrak{p}_K$. Let $\varpi_K$ be the uniformizer of $K$. Let $k_K$ denote the residue field $\mathfrak{o}_K/\mathfrak{p}_K$, and let $q_K$ be the cardinality of $k_K$. Let $\upsilon_K$ be the normalized discrete valuation on $K$, and let $\nu_K$ be the absolute value on $K$ corresponding to $\upsilon_K$. For each positive integer $n$, we denote by $G_n(K)$ the group of all $n\times n$ invertible matrices with entries in $K$.

We fix an algebraic closure $\overline{\mathbb{Q}}_l$ of $\mathbb{Q}_l$. Let $\overline{\mathbb{Z}}_l$ be the integral clsoure of $\mathbb{Z}_l$ in $\overline{\mathbb{Q}}_l$. Let $\Lambda$ denote the ring of Witt vectors of $\overline{\mathbb{F}}_l$--which is, by definition, the $l$-adic completion of the ring of integers of the maximal unramified extension of $\mathbb{Q}_l$ in $\overline{\mathbb{Q}}_l$. 
Throughout this article, we deal with an important class of Noetherian $\Lambda$-algebras, namely Artin local $\overline{\mathbb{F}}_l$-algebras. By definition, an Artin local $\overline{\mathbb{F}}_l$-algebra $R$ is a finite dimensional local $\overline{\mathbb{F}}_l$-algebra with unique prime ideal, denoted by $\mathfrak{m}_R$, such that the quotient $R/\mathfrak{m}_R$ is isomorphic to $\overline{\mathbb{F}}_l$ and the composition $\overline{\mathbb{F}}_l\rightarrow R\rightarrow\overline{\mathbb{F}}_l$ is the identity map. 

\subsection{}\label{pre_1} 
Let $G$ be a locally compact and totally disconnected group. Let $R$ be a Noetherian ring. Let $(\pi,V)$ be a smooth $R$-representation of $G$ which means $V$ is an $R$-module and every vector in $V$ is stabilized by a compact open subgroup of $G$. The representation $\pi$ is called {\it admissible} if, for each compact open subgroup $J$ of $G$, the set of $J$-fixed vectors $V^J$ is finitely generated as an $R$-module. In this paper, all representations are assumed to be smooth. An $R$-representation of $G$ is also called an $R[G]$-module. The representation $(\pi,V)$ is called $l$-adic if $R=\overline{\mathbb{Q}}_l$, and $(\pi,V)$ is called $l$-modular if $R=\overline{\mathbb{F}}_l$. For any closed subgroup $H$ of $G$, let ${\rm Ind}_H^G$ and ${\rm ind}_H^G$ be the smooth induction functor and compact induction functor respectively.

\subsubsection{Cuspidal representation}
An irreducible smooth $R$-representation $\pi$ of $G_n(K)$ is called {\it cuspidal} (resp. {\it supercuspidal}) if $\pi$ does not occur as a quotient (resp. sub--quotient) of parabolically induced representation ${\rm ind}_P^{G_n(K)}(\tau)$, for all proper parabolic subgroup $P$ of $G_n(K)$ and all irreducible $R$-representation $\tau$ of the Levi subgroup of $P$.

When $R=\overline{\mathbb{Q}}_l$, then $\pi$ is cuspidal 
if and only if $\pi$ is supercuspidal. But when $R=\overline{\mathbb{F}}_l$, there are cuspidal 
representations of $G_n(K)$ which are not supercuspidal. For details, see \cite[Section 2.5,
Chapter 2]{Vigneras-l-modular}.

\subsubsection{Generic representation}\label{gen_rep}
Let $\mathbb{Q}_l^{un}$ denote the maximal unramified extension of $\mathbb{Q}_l$ in $\overline{\mathbb{Q}}_l$ with ring of integers $\mathbb{Z}_l^{un}$. Then the ring of Witt vectors $\Lambda$ is the $l$-adic completion of $\mathbb{Z}_l^{un}$. Fix a non-trivial additive character $\psi_K : K\rightarrow (\mathbb{Z}_l^{un})^\times$. For a Noetherian $\mathbb{Z}_l^{un}$-algebra $R$, the composition $K\xrightarrow{\psi_K} \Lambda^\times\rightarrow R^\times$ is also denoted as $\psi_K$, by abuse of notation. Let $N_n(K)$ be the group of unipotent upper triangular matrices in $G_n(K)$. Let $\Theta_K$ be the non-degenerate character of $N_n(K)$, defined as 
\begin{equation}\label{non_deg_char}
\Theta_K((x_{ij})_{i,j=1}^n) :=
\psi_K(x_{12} + ....+ x_{n(n-1)}),
\end{equation}
for $(x_{ij})_{i,j=1}^n\in N_n(K)$.
We denote by $\overline{\psi}_K$ and $\overline{\Theta}_K$ the mod-$l$ reductions of $\psi_K$ and $\Theta_K$, respectively. A smooth $R$-representation $(\pi,V)$ is called {\it Whittaker type} if the space
$$
{\rm Hom}_{G_n(K)}\big(\pi, {\rm Ind}_{N_n(K)}^{G_n(K)}(\Theta_K)\big)
$$
is a free $R$-module of rank one. An $l$-adic (or $l$-modular) representation $\pi$ of $G_n(K)$ is called generic if $\pi$ is irreducible and is of Whittaker type. 
We denote by $\mathcal{A}_{C,K}^{gen}(n)$ the set of isomorphism classes of generic $C$-representations of $G_n(K)$, where $C=\overline{\mathbb{Q}}_l$ or $\overline{\mathbb{F}}_l$.

\subsubsection{Whittaker model and Kirillov model}
Let $(\pi,V)$ be an $R$-representation of $G_n(K)$. We assume that $(\pi,V)$ is of Whittaker type. Using Frobenius reciprocity, we get that the space ${\rm Hom}_{N_n(K)}(\pi,\Theta_K)$ is a free $R$-module of rank one. Let $\mathcal{W}_{\pi}$ be a non-zero element in ${\rm Hom}_{N_n(K)}(\pi,\Theta_K)$. Let
$\mathbb{W}(\pi,\psi_K)\subseteq{\rm Ind}_{N_n(K)}^{G_n(K)}(\Theta_K)$
be the space consisting of functions $W_v$, $v \in V$, 
where
$$
W_v(g):=\mathcal{W}_{\pi}\big(\pi(g)v\big),
$$
for $g \in G_n(K)$. The map $v\mapsto W_v$ induces a $G_n(K)$ equivariant surjection from $(\pi,V)$ to $\mathbb{W}(\pi,\psi_K)$. The space $\mathbb{W}(\pi,\psi_K)$ is called an $R$-valued {\it Whittaker model} of $(\pi,V)$ with respect to $\psi_K$.

Consider the space $\mathbb{K}(\pi,
\psi_K)$ of all elements $W$ restricted to $P_n(K)$, where
$W$ varies over $\mathbb{W}(\pi,
\psi_K)$. Then the space
$\mathbb{K}(\pi,\psi_K)$ is
$P_n(K)$ invariant, and is called the {\it Kirillov model} of $\pi$.
Moreover, ${\rm ind}_{N_n(K)}^{P_n(K)}
(\Theta_K) \subseteq \mathbb{K}(\pi,
\psi_K)$. When $R$ is a field, then $\mathbb{K}(\pi,
\psi_K) = {\rm ind}_{N_n(K)}
^{P_n(K)}(\Theta_K)$ if and only if $\pi$ is cuspidal. 

\subsubsection{Co-Whittaker modules}
Let $(\pi,V)$ be a smooth $R$-representation of $G_n(K)$. Let $V_{N_n(K),\Theta_K}$ be the quotient space $V/V(N_n(K),\Theta_K)$, where the space $V(N_n(K),\Theta_K)$ is the $R$-submodule of $V$ generated by the following set of vectors: 
$$ \big\{\pi(x)v-\Theta_K(x)v:x\in N_n(K), v\in V\big\}. $$ 
The representation $(\pi,V)$ is called {\it co-Whittaker} if the following holds:
\begin{enumerate}
	\item $(\pi,V)$ is admissible,
	\item $V_{N_n(K),\Theta_K}$ is a free $R$-module of rank one,
	\item If $W$ is a quotient of $V$ with $W^{(n)}=0$, then $W=0$.
\end{enumerate}
By definition, a co-Whittaker $R$-representation $(\pi,V)$ of $G_n(K)$ possesses an $R$-valued Whittaker model. Moreover, the representation $(\pi,V)$ admits a central character, denoted by $\omega_\pi$.

\subsubsection{Nilpotent lifts}\label{nil_GLn}
Let $R$ be an Artin local $\overline{\mathbb{F}}_l$-algebra. Let $\pi$ be an $l$-modular generic representation of $G_n(K)$. A {\it nilpotent lift} of $\pi$ to the algebra $R$ is an admissble $R[G_n(K)]$-submodule $\widetilde{\pi} \subseteq {\rm Ind}_{N_n(K)}^{G_n(K)}(\Theta_K)$ such that 
$$
\widetilde{\pi}\otimes_R\overline{\mathbb{F}}_l
\simeq \mathbb{W}(\pi,\overline{\psi}_K). 
$$
For an Artin local $\overline{\mathbb{F}}_l$-algebra $R$, we denote by $\mathcal{A}_{R,K}^{{nil}}(n)$ the set of isomorphism classes of all nilpotent lifts of objects in $\mathcal{A}_{\overline{\mathbb{F}}_l,K}^{gen}(n)$. For $n=1$, the objects in $\mathcal{A}_{\overline{\mathbb{F}}_l,K}^{gen}(1)$ are the $\overline{\mathbb{F}}_l$-valued characters of $K^\times$. The nilpotent lift of one such character, say $\overline{\chi}$, to $R$ is given by a free $R$-module $\mathcal{M}$, on which $K^\times$ acts via an $R$-valued character $\chi$ of $K^\times$ such that the reduction of $\chi$ modulo $\mathfrak{m}_R$ is equal to $\overline{\chi}$.

\section{Rankin-Selberg gamma factors over arbitrary rings}\label{Rankin_Selberg_gamma_factors}
In this section, we review the Rankin--Selberg theory over arbitrary Noetherian $\Lambda$-algebras for the pair $(G_3(K), K^\times)$. We refer to \cite{Moss_Gammafactors}, \cite{Helm_Co-whittaker} for the general theory.
\subsection{}
We now introduce some notations. For each $r\in\mathbb{Z}$, we denote by $G_n^r(K)$ the following set
$$ \big\{g\in G_n(K) : \upsilon_K(\det(g)) = r\big\}. $$ 
We let $X_K$ be the right coset space $N_n(K) \setminus G_n(K)$. For each integer $r$, let $X_K^r$ denote the coset space $\big\{N_n(K)g: g\in G_n^r(K)\big\}$. In particular, for $n=1$, we have $X_K^r=\varpi_K^r\mathfrak{o}_K^\times$.
We denote by $w_n$ and $w_{n,m}$ the following matrices of $G_n(K)$:
\begin{center} $w_n =
\begin{pmatrix}
	0  &  & &  1  \\
	&  & . \\
	& .   \\
	1 &  &  & 0
\end{pmatrix}$,\,\, $w_{n,m}=
\begin{pmatrix}
	I_{n-m} & 0\\
	0 & w_m
\end{pmatrix}$. 
\end{center}
\subsection{Rankin--Selberg functional equations}
Let $A$ and $B$ be two Noetherian $\Lambda$-algebras, and let $R=A\otimes_{\Lambda}B$. Let $\pi$ be a $A[G_3(K)]$-module of Whittaker type and let $\chi$ be a $B$-valued character of $K^\times$. For $W \in \mathbb{W}(\pi,\psi_K)$, the integrals
$$
c_r^K(W,\chi;0)=\int_{\varpi_K^r\mathfrak{o}_K^\times}W
\begin{pmatrix}
	g & 0 & 0\\
	0 & 1 & 0\\
	0 & 0 & 1
\end{pmatrix}
\otimes\chi(g)\,dg
$$
and
$$
c_r^K(W,\chi;1)=\int_K\int_{\varpi_K^r
	\mathfrak{o}_K^\times}W
\begin{pmatrix}
	g & 0 & 0\\
	x & 1 & 0\\
	0 & 0 & 1
\end{pmatrix}
\otimes\chi(g)\,dg\,dx$$
are well-defined. Then, for each $j\in\{0,1\}$, the series
\begin{equation}\label{RS_series}
	\sum_{r\in\mathbb{Z}}c_r^K(W,\chi;j)X^r
\end{equation}
is an element of $R[[X]][X^{-1}]$, and it is called the Rankin-Selberg series associated with the pair $(\pi,\pi')$. In addition, if $\pi$ is admissible and finitely generated over $A[G_3(K)]$, then the series (\ref{RS_series}) lies in $S^{-1}(R[X,X^{-1}])$ (see \cite[Theorem 3.2]{Moss_Gammafactors}), where $S$ is the multiplicative subset of $R[X,X^{-1}]$, consisting of polynomials of the form $\sum_{r=-s}^ta_rX^r$ with $a_s$ and $a_t$ being units in $R$. For any function $f$ on $G_n(K)$, let $\widetilde{f}$ be the function on $G_n(K)$, defined by $\widetilde{f}(g)=f(w_n(g^{-1})^t)$, for $g\in G_n(K)$. Then we have the following result:
\begin{theorem}\cite[Theorem 5.3]{Moss_Gammafactors}\label{func_equation}
Let $A$ and $B$ be two Noetherian $\Lambda$-algebras. Let $\pi$ be a co-Whittaker $A[G_3(K)]$ module, and let $\eta$ be a $B$-valued character of $K^\times$. Then, there exists a unique element $\gamma_R(X,\pi,\eta,\psi_K)$ in $S^{-1}(R[X,X^{-1}])$ such that
\begin{equation}\label{FE_1}
\sum_{r\in\mathbb{Z}}c_r^K\big(w_{3,1}\widetilde{W},
\eta^{-1};1\big)q_K^rX^{-r} = \gamma_R(X,\pi,\eta,\psi_K)
\sum_{r\in\mathbb{Z}}c_r^K(W,\eta;0)X^r
\end{equation}
and
\begin{equation}\label{FE_2}\sum_{r\in\mathbb{Z}}c_r^K
\big(w_{3,1}\widetilde{W},\eta^{-1}
;0\big)q_K^rX^{-r} = \gamma_R(X,\pi,\eta,\psi_K)
\sum_{r\in\mathbb{Z}}c_r^K(W,\eta;1)X^r,
\end{equation}
for all $W\in\mathbb{W}(\pi,\psi_K)$. 
\end{theorem}

\subsubsection{}
In \cite[Section 2]{moss_mod_l_nilpotent_gamma}, the author shows that the local converse theorem, which holds for generic representations defined over algebraically closed fields of characteristic zero, is no longer true for generic representations defined over $\overline{\mathbb{F}}_l$. To be more precise, let $\pi_1$ and $\pi_2$ be two $l$-modular generic representations for $G_n(K)$. Suppose we have the following equality of $\gamma$-factors
$$ \gamma_{\overline{\mathbb{F}}_l}
(X,\pi_1,\tau,\overline{\psi}_K) =
\gamma_{\overline{\mathbb{F}}_l}
(X,\pi_2,\tau,\overline{\psi}_K), $$ 
for all $\tau\in\mathcal{A}_{\overline{\mathbb{F}}_l}^{gen}(t)$, $1\leq t\leq n-1$. Then, it may happen that $\pi_1$ is not isomorphic to $\pi_2$. Moss (\cite[Theorem 1.2]{moss_mod_l_nilpotent_gamma}) proved that the local converse theorem can be recovered if we vary $\tau$ over $\mathcal{A}_{R,K}^{nil}(t)$ for all $t\in\{1,2,\dots,[\frac{n}{2}]\}$, and for all Artin local $\overline{\mathbb{F}}_l$-algebras $R$. We now state the theorem for $n=3$.
\begin{theorem}\label{converse_thm}
Let $\pi_1$ and $\pi_2$ be two $l$-modular generic representations of $G_3(K)$. Suppose
$$
\gamma_R(X,\pi_1,\eta,\overline{\psi}_K) = 
\gamma_R(X,\pi_2,\eta,\overline{\psi}_K),
$$ 
for all nilpotent lifts $\eta\in\mathcal{A}_{R,K}^{nil}(1)$ and for all Artin local $\overline{\mathbb{F}}_l$-algebras $R$. Then $\pi_1$ is isomorphic to $\pi_2$.
\end{theorem} 

\subsection{Frobenius twist}\label{Frob_twist}
Let $G$ be a locally profinite group. Let $(\rho,V)$ be an $l$-modular representation of $G$. Let ${\rm Frob}$ be the Frobenius automorphism of $\overline{\mathbb{F}}_l$. The {\it Frobenius twist} of $\rho$, denoted by $\rho^{(l)}$, is defined as an $l$-modular representation of $G$ on the space $V \otimes_{\rm Frob}\overline{\mathbb{F}}_l$, where the underlying additive group structure of the vector space $V \otimes_{\rm Frob}\overline{\mathbb{F}}_l$ is same as that of $V$ but the scalar action $*$ on $V \otimes_{\rm Frob}\overline{\mathbb{F}}_l$ is given by
\begin{center}
$\lambda*v = \lambda^{\frac{1}{l}}v$, for $\lambda\in\overline{\mathbb{F}}_l$, $v\in V$.
\end{center}
Note that the action of $G$ on $V$ induces the representation $(\rho^{(l)}, V \otimes_{\rm Frob}\overline{\mathbb{F}}_l)$. 
We conclude this section with the following lemma, which will be used in the proof of the main result. 
\begin{lemma}\label{lemma_frob}
Let $R$ be an Artin local $\overline{\mathbb{F}}_l$-algebra, and let $\eta$ be an $R$-valued character of $K^\times$. Let $(\pi,V)$ be an $l$-modular generic representation of $G_3(K)$. Then we have
$$
\gamma_R(X,\pi,\eta,\overline{\psi}_K)^l=\gamma_R(X^l,\pi^{(l)},
\eta^l,\overline{\psi}_K^l).
$$
\end{lemma}
\begin{proof}
Let $\mathcal{W}_\pi$ be a Whittaker functional on the representation
$\pi$ and let $\mathcal{W}_{\pi^{(l)}}$ be the composite map
$$
V\xrightarrow{\mathcal{W}_\pi}\overline{\mathbb{F}}_l
\xrightarrow{x \mapsto x^l}\overline{\mathbb{F}}_l.
$$ 
Then $\mathcal{W}_{\pi^{(l)}}$ is a Whittaker functional on the representation $\pi^{(l)}$ with respect to the character $\overline{\psi}_K^l$, and the Whittaker space $\mathbb{W}(\pi^{(l)},\overline{\psi}_K^l)$ consists of the functions $W_v^l$, where $W_v$ varies in $\mathbb{W}(\pi,\overline{\psi}_K)$. Using functional equations (\ref{FE_1}) and (\ref{FE_2}), we have
\begin{equation}\label{CM_1}
\sum_{r\in\mathbb{Z}}c^K_r(\widetilde{W_v},\eta^{-1};1-j)
^lq_K^{lr}X^{-lr} =
\gamma_R(X,\pi,\eta,\overline{\psi}_K)^l\sum_{r \in\mathbb{Z}}c^K_r(W_v,\eta;j)^lX^{lr}
\end{equation}	
and 
\begin{equation}\label{CM_2}
\sum_{r\in\mathbb{Z}}c^K_r(\widetilde{W_v^l},\eta^{-l};1-j)
q_K^rX^{-r}=\gamma_R(X,\pi^{(l)},\eta^l,\overline{\psi}_K^l)
\sum_{r\in \mathbb{Z}}c^K_r(W_v^l,\eta^l;j)X^r,
\end{equation}
for each $j\in\{0,1\}$. Replacing $X$ by $X^l$ to the equation (\ref{CM_2}), we get
$$
\sum_{r\in\mathbb{Z}}c^K_r(\widetilde{W_v^l},\eta^{-l};1-j)
q_K^rX^{-lr} = \gamma_R(X^l,\pi^{(l)},\eta^l,\overline{\psi}_K^l)
\sum_{r\in \mathbb{Z}}c^K_r(W_v^l,\eta^l;j)X^{lr}.
$$
Finally, it follows from the above equation and the equality (\ref{CM_1}) that
$$ 
\gamma_R(X,\pi,\eta,\overline{\psi}_K)^l = 
\gamma_R(X^l,\pi^{(l)},\eta^l,\overline{\psi}_K^l). 
$$
\end{proof}	

\section{Representations of Weil-Deligne group}
In this section, we review the theory of Weil--Deligne representations over arbitrary Noetherian $\Lambda$-algebras and the associated Deligne--Langlands $\gamma$-factors. For precise reference, see \cite{helm2015deligne}.
\subsection{}
We fix an algebraic closure $K_s$ of $K$. Let $\mathcal{W}_K\subseteq{\rm Gal}(K_s/K)$ be the Weil group of $K$. For any finite seperable extension $L$ of $K$ contained in $K_s$, the Weil group $\mathcal{W}_L$ is considered as a subgroup of $\mathcal{W}_K$. Let $R$ be a Noetherian $\Lambda$-algebra.
\subsubsection{}
For each non-negetive integer $m$, we denote by $R_m$ the kernel of the canonical map ${\rm GL}_n(R)\rightarrow {\rm GL}_n(R/l^mR)$. A representation $\rho:\mathcal{W}_K\rightarrow {\rm GL}_n(R)$ is called {\it $l$-adically continuous} if for all $m$, the preimage set $$\rho^{-1}(R_m)=\{x\in\mathcal{W}_K:\rho(x)\in R_m\}$$ is open in $\mathcal{W}_K$. In particular, any smooth $R[\mathcal{W}_K]$ module $\rho$ is $l$-adically continuous. We denote by $\mathcal{G}_{\overline{\mathbb{F}}_l,K}^{ss}(n)$ the set of isomorphism classes of $n$-dimensional, smooth, semisimple $l$-modular representations of $\mathcal{W}_K$.
\subsubsection{Nilpotent lifts}\label{nil_Weil}
Let $R$ be an Artin local $\overline{\mathbb{F}}_l$-algebra. Let $\rho$ be a smooth, $n$-dimensional, semisimple $l$-modular representation of $\mathcal{W}_K$. A {\it nilpotent lift} of $\rho$ to $R$ is a smooth $R[\mathcal{W}_K]$ module $\widetilde{\rho}$ such that $\widetilde{\rho}$ is a free $R$-module of rank $n$ and there exists an isomorphism
$$
\widetilde{\rho}\otimes_R\overline{\mathbb{F}}_l
\simeq\rho.
$$
For any Artin local $\overline{\mathbb{F}}_l$-algebra $R$, we denote by $\mathcal{G}_{R,K}(n)$ the set of isomorphism classes of nilpotent lifts of objects in $\mathcal{G}_{\overline{\mathbb{F}}_l,K}^{ss}(n)$.
\subsection{Deligne-Langlands gamma factors in families (\cite[Theorem 1.1]{helm2015deligne})}
Let $R$ be a Noetherian $\Lambda$-algebra. Let $S$ be the multiplicative subset of $R[X,X^{-1}]$  consisting of Laurent polynomials whose first and last coefficients are units. Then, for any $l$-adically continuous representation $\rho:\mathcal{W}_K\rightarrow {\rm GL}_n(R)$, there exists a unique element $\gamma_R(X,\rho,\psi_K)\in S^{-1}(R[X,X^{-1}])$ with the following properties:
\begin{enumerate}
\item For any Noetherian $\Lambda$-algebra $R'$ and a morphism $f:R\rightarrow R'$, we have
$$ 
f(\gamma_R(X,\rho,\psi_K))=\gamma_{R'}(X,\rho\otimes_R R',\psi_K). 
$$
\item In particular, if $R$ is the field of fractions of a complete local domain of characteristic zero with residue field $\overline{\mathbb{F}}_l$, then $\gamma_R(X,\rho,\psi_K)$ coincides with the classical Deligne--Langlands $\gamma$-factor.
\end{enumerate}
Moreover, the association $(R,\rho,\psi_K)\longmapsto\gamma_R(X,\rho,\psi_K)$ satisfies:
\begin{enumerate}
\item For any exact sequence $0\rightarrow\rho_1\rightarrow\rho_2\rightarrow\rho_3
\rightarrow 0$ of representations of $\mathcal{W}_K$, 
we have
$$ 
\gamma_R(X,\rho_2,\psi_K) = \gamma_R(X,\rho_1,\psi_K)\,
\gamma_R (X,\rho_3,\psi_K). 
$$
\item If $K_1$ is a finite seperable extension of $K$ and $\delta_{K_1}$ is a virtual smooth representation of $\mathcal{W}_{K_1}$ of virtual degree zero, then
$$
\gamma_R\big(X,\delta_{K_1},\psi_K\circ{\rm Tr}_{K_1/K}\big)=\gamma_R\big(X,{\rm ind}_{\mathcal{W}_{K_1}}^{\mathcal{W}_K}(\delta_{K_1}),
\psi_K\big).
$$
\end{enumerate}

\section{Local Langlands Correspondence}
\subsection{The mod-$l$ local Langlands correspondence}\label{mod-l_LLC}
In this subsection, we recall $l$-adic and mod-$l$ local langlands correspondence. We fix a square root of $q_F$, say $q_F^{1/2}$, in
$\overline{\mathbb{Q}}_l$. The choice of $q_F^{1/2}$ and the twisting by the unramified character $\nu_F^{(1-n)/2}$ is required for transferring the local Langlands correspondence over $\mathbb{C}$ to a local Langlands correspondence over $\overline{\mathbb{Q}}_l$ (see \cite[Section 7]{henniart_bordeaux} and for $n=2$, see \cite[Section 35.1, Chapter 8]{BH_LLC_GL2} for a reference). To be more precise, for each $n\geq 1$, there is a bijective correspondence
$$
L_{\overline{\mathbb{Q}}_l,n}^{K,0}:
\mathcal{G}_{\overline{\mathbb{Q}}_l,K}^0(n)
\longrightarrow\mathcal{A}_{\overline{\mathbb{Q}}_l,K}^0(n),
$$ 
and it induces a bijection
$L^K_{\overline{\mathbb{Q}}_l,n}:
\mathcal{G}_{\overline{\mathbb{Q}}_l,K}^{ss}(n)
\rightarrow\mathcal{A}_{\overline{\mathbb{Q}}_l,K}^{gen}(n)$. In the seminal work \cite[Theorem 1.6]{Vigneras_semisimple}, Vign\'eras proved that there exists a bijection
$$
L^{K,0}_{\overline{\mathbb{F}}_l,n}:
\mathcal{G}_{\overline{\mathbb{F}}_l,K}^0(n)
\longrightarrow\mathcal{A}_{\overline
{\mathbb{F}}_l,K}^0(n),
$$
for each $n\geq 1$, which extends to a one-to-one correspondence
$L_{\overline{\mathbb{F}}_l,n}^K:
\mathcal{G}_{\overline{\mathbb{F}}_l,K}^{ss}(n)
\longrightarrow\mathcal{A}_{\overline
{\mathbb{F}}_l,K}^{gen}(n)$. Moreover, the sequence $\big\{L_{\overline{\mathbb{F}}_l,n}^K\big\}_{n\geq 1}$ is compatible with $\big\{L_{\overline{\mathbb{Q}}_l,n}^{K,0}\big\}_{n\geq 1}$ under the reduction mod-$l$. 
\subsubsection{Local Langlands correspondence on nilpotent lifts}\label{LLC_nilpoent}
In the article \cite[Theorem 1.4]{moss_mod_l_nilpotent_gamma}, Moss gave a characterization of the sequence of bijections $\big\{L_{\overline{\mathbb{F}}_l,n}^K\big\}_{n\geq 1}$ via $\gamma$-factors by extending $\big\{L_{\overline{\mathbb{F}}_l,n}^K\big\}_{n\geq 1}$ to a correspondence on nilpotent lifts. 
\begin{theorem}\label{char_modl_thm}
For each $n\geq 1$, there is a unique bijection $$
L_{\overline{\mathbb{F}}_l,n}^K:
\mathcal{G}_{\overline{\mathbb{F}}_l,K}^{ss}(n)
\longrightarrow\mathcal{A}_{\overline{\mathbb{F}}_l,K}^{gen}(n),
$$
such that for every Artin-local $\overline{\mathbb{F}}_l$-algebra $R$, it admits an extension to a sequence of maps 
$$
L^K_{R,n}:\mathcal{G}_{R,K}(n)\longrightarrow
\mathcal{A}_{R,K}^{nil}(n),
$$ 
with the property that for all $n>m$, $\rho\in\mathcal{G}_{R,K}(n)$ and $\rho'\in\mathcal{G}_{R',K}(m)$, we have 
\begin{equation}\label{LLC}
\gamma\big(X,L_{R,n}^K(\rho),L^K_{R',m}(\rho'),\psi\big)
= \gamma(X,\rho\otimes\rho',\psi)
\end{equation} 
in $(R\otimes_{\overline{\mathbb{F}}_l}R')[[X]][X^{-1}]$. In particular, the map $L^K_{R,1}$ is given by the local class field theory and it is a bijection.
\end{theorem}

\subsection{Local base change for cyclic extension}\label{local_base_change}
Let $F$ be a finite extension of $\mathbb{Q}_p$, and 
let $E$ be a Galois extension of $F$ of prime degree $l$, where $l\ne p$. The base change lifting on irreducible smooth representations of $G_n(F)$ over complex vector spaces is characterized by
certain character identities (\cite[Chapter 3]{Arthur_Clozel_BC}). In this subsection, we recall the
relation between the local Langlands correspondence and the local base change lifting of $l$-adic generic representations of $G_n(F)$ to $G_n(E)$. 
\subsubsection{}
Let $\pi_F$ be an $l$-adic generic representation of $G_n(F)$. Let $\pi_E$ be the $l$-adic generic representation of $G_n(E)$ such that
\begin{equation}\label{base_change_def}
{\rm res}_{\mathcal{W}_E}
((L_{\overline{\mathbb{Q}}_l,n}^F)^{-1}
(\pi_F))\simeq (L_{\overline
	{\mathbb{Q}}_l,n}^E)^{-1}(\pi_E).
\end{equation}
The representation $\pi_E$ is called  the base change lifting of $\pi_F$. In this case, $\pi_E$ is isomorphic to $\pi_E^\gamma$, for all $\gamma\in{\rm Gal}(E/F)$. Moreover, we have the following lemma about the base change lifting and integrality.
\begin{lemma}\label{bc_integral}
Let $\pi_F$ be an $l$-adic generic representation of $G_n(F)$, and let $\pi_E$ be the base change lifting of 
$\pi_F$ to $G_n(E)$. Then $\pi_F$ is integral if and only if the base change lifting $\pi_E$ is integral.
\end{lemma}
\begin{proof}
Let $\pi$ be an irreducible, smooth, $l$-adic
representation of $G_n(K)$. Let ${\rm scs}(\pi)$
be the supercuspidal support of $\pi$ (see
\cite[III.3.]{Vigneras_Induced} for the definition). The
representation $\pi$ is integral
if and only if ${\rm scs}(\pi)$ is integral (see \cite[Section 1.4]{Vigneras_semisimple}, and for
general reductive groups, see \cite[Corollary
1.6]{dat2024finiteness} for a reference).
The representation ${\rm scs}(\pi)$ is integral if and only if
the central character of ${\rm scs}(\pi)$ is
integral. 
	
Let us first assume that $\pi_F$ is integral. Let $\rho_F$ be the semisimple, $l$-adic representation of $\mathcal{W}_F$ associated with $\pi_F$ under the
local Langlands correspondence (LLC) $L^F_{\overline{\mathbb{Q}}_l,n}$. Under LLC,
the representation ${\rm scs}(\pi_F)$ corresponds to the
$\mathcal{W}_F$-representation $\rho_F$.  Since the determinant character of each irreducible
component of $\rho_F$ is integral, we get that $\rho_F$ is
integral. This implies that the restriction
${\rm res}_{\mathcal{W}_E}(\rho_F)$ is also integral. Under 
LLC, the supercuspidal support of $\pi_E$ corresponds to the restriction
${\rm res}_{\mathcal{W}_E}(\rho_F)$. 
Thus, the supercuspidal support of $\pi_E$ is integral
and we get that 
$\pi_E$ is integral.
	
Conversely, assume that $\pi_E$ is integral. Let
$\rho_E$ be the semisimple, $l$-adic
representation of $\mathcal{W}_E$ associated with $\pi_E$ under
LLC. Since the representation ${\rm scs}(\pi_E)$ is integral, the
$\mathcal{W}_E$-representation $\rho_E$ (which is
${\rm res}_{\mathcal{W}_E}(\rho_F)$) is integral. Let $\mathcal{L}$
be a $\mathcal{W}_E$-invariant lattice in $\rho_E$. Then,
$$ \sum_{x\in \mathcal{W}_F/\mathcal{W}_E} \rho_F(x)\mathcal{L}. $$ 
is a $\mathcal{W}_F$-invariant lattice in $\rho_F$. Thus, the $l$-adic representation $\rho_F$ is also integral, which implies that the
supercuspidal support ${\rm scs}(\pi_F)$ is integral. Hence, $\pi_F$ is integral.
\end{proof}

\begin{remark}\label{mod_l_bc_form}
\rm Let $\pi_F$ be an integral, $l$-adic generic representation of $G_n(F)$, and let $\pi_E$ be the base change lifting of $\pi_F$ to $G_n(E)$. Then $\pi_E$ is generic and is also integral by Lemma \ref{bc_integral}. Using the isomorphism (\ref{base_change_def}), we get 
\begin{equation}\label{base_change_mod_l}
{\rm res}_{\mathcal{W}_E}\big(r_l((L_{
\overline{\mathbb{Q}}_l,n}^F)^{-1} (\pi_F))\big)\simeq r_l\big((L_{\overline{\mathbb{Q}}_l,n}^E)^{-1}
(\pi_E)\big)
\end{equation} 
Let $\big\{\tau_F^1,\tau_F^2,\dots,\tau_F^r\big\}$ and $\big\{\tau_E^1,\tau_E^2,\dots,\tau_E^s\big\}$ be the supercuspidal supports of $J_l(\pi_F)$ and $J_l(\pi_E)$ respectively, where $1\leq r,s\leq n$. 
Let $\rho_F^i$ (resp. $\rho_E^j$) be the semisimple, $l$-modular representation of $\mathcal{W}_F$ (resp. $\mathcal{W}_E$) associated with $\tau_F^i$ (resp. $\tau_F^j$) under the mod-$l$ local Langlands correspondence $L^F_{\overline{\mathbb{F}}_l,n}$ (resp. $L^E_{\overline{\mathbb{F}}_l,n}$). Then, it follows from (\ref{base_change_mod_l}) that
$$
{\rm res}_{\mathcal{W}_E}(\bigoplus_{i=1}^r \rho_F^i)\simeq 
\bigoplus_{j=1}^s \rho_E^j.
$$
\end{remark}

\section{Tate cohomology}
Let $X$ be an $l$-space (i.e., locally compact and totally disconnected) equipped with an action of a cyclic group $\langle\gamma\rangle$ of finite order. In this section, we recall the Tate cohomology groups of $\gamma$-equivariant sheaves of $\Lambda$-modules on $X$. For a reference, see \cite[Section 3]{Venkatesh-Treumann}. Then we put up some initial results, which will be needed later. 

\subsection{}
Let $X$ be an $l$-space equipped with an action of a finite cyclic group $\langle\gamma\rangle$ of order $l$. For an $l$-sheaf of $\Lambda$-modules  $\mathcal{F}$ on $X$, we denote by $\Gamma_c(X;\mathcal{F})$ the space of compactly supported sections of $\mathcal{F}$. In particular, if $\mathcal{F}$ is the constant sheaf with stalk $\Lambda$, then
$\Gamma_c(X;\mathcal{F})=C_c^\infty(X;\Lambda)$. In general, if $\mathcal{F}$ is $\langle\gamma\rangle$-equivariant, then $\gamma$ induces a map of restricted sheaves $\mathcal{F}|_{X^\gamma} \rightarrow \mathcal{F}|_{X^\gamma}$, and the Tate cohomology groups are defined as
\begin{center}
$\widehat{H}^0(\mathcal{F}|_{X^\gamma}) := 
\dfrac{{\rm Ker}(1-\gamma)}{{\rm Img}(N)}$ and
$\widehat{H}^1(\mathcal{F}|_{X^\gamma}) := 
\dfrac{{\rm Ker}(N)}{{\rm Img}(1-\gamma)}$,
\end{center}
where $N = 1+\gamma +\dots +\gamma^{l-1}$, viewed as an element of the group ring $\Lambda[\langle\gamma\rangle]$. A compactly supported section of $\mathcal{F}$ can be restricted to a
compactly supported section of $\mathcal{F}|_{X^\gamma}$, and we have the following result:
\begin{proposition}\cite[Proposition 3.3]{Venkatesh-Treumann}\,\label{TV_isom}
For each $i\in\{0,1\}$, the restriction map induces an isomorphism 
\begin{center}
$\widehat{H}^i(\Gamma_c(X; \mathcal{F}) \xrightarrow{\sim} \Gamma_c(X^\gamma; \widehat{H}^i(\mathcal{F}))$.
\end{center}
\end{proposition}

\subsection{Comparison of integrals}
Let $R$ be a local Artinian $\overline{\mathbb{F}}_l$-algebra, and let $\mathfrak{m}_R$ be the unique maximal ideal of $R$. Let $X$ be a non-empty set, equipped with an action of the cyclic group $\Gamma=\langle\gamma\rangle$. The space $C_c^\infty(X;R)$ of all locally constant and compactly supported $R$-valued functions on $X$ has a natural action of $\Gamma$, defined by
$$ (\gamma . f)(x) := f(\gamma^{-1}x), $$
for all $x\in X$ and $f \in C_c^\infty(X;R)$. Let $C_c^\infty(X;R)^\Gamma$ be the space of all $\Gamma$-invariant functions in $C_c^\infty(X;R)$. For two positive integers $r$ and $s$, let $Y_F$ and $Y_E$ denote the ring of $r\times s$ matrices with entries in $F$ and $E$ respectively. Recall that $X_K$ denotes the coset space $N_n(K)\setminus G_n(K)$. Then we have the following result:
\begin{proposition}\label{com_int_1}
Let $d\mu_E$ (resp. $d\nu_E$) be the Haar measure on $X_E$ (resp. $Y_E$), and let $d\mu_F$ (resp. $d\nu_F$) be the Haar measure on $X_F$ (resp. $Y_F$). Then, there exists $c\in R^\times$ such that
$$
\int_{Y_E}\int_{X_E}\varphi\,d\mu_E\,d\nu_E = c\int_{Y_F}\int_{X_F}\varphi\,d\mu_F\,d\nu_F,
$$
for all $\varphi\in C_c^\infty(Y_E\times X_E;R)^\Gamma$.
\end{proposition}
\begin{proof}
Since $N_n(E)$ is stable under the action of $\Gamma$ on
$G_n(E)$, we have the following long exact sequence of
non-abelian cohomology (\cite[Chapter VII, Appendix]{Serre-local-field}):
$$
0\longrightarrow N_n(E)^\Gamma\longrightarrow G_n(E)^\Gamma\longrightarrow X_E^\Gamma\longrightarrow H^1\big(\Gamma; N_n(E)\big)\longrightarrow H^1\big(\Gamma; G_n(E)\big).
$$
Since the pointed set $H^1\big(\Gamma,N_n(E)\big)$ is trivial, the above long exact sequence gives the equality between the sets $X_E^\Gamma$ and $X_F$.
Note that $X_F$ (resp. $Y_F$) is closed in $X_E$ (resp. $Y_E$). Then, we have the following exact sequence of $\Gamma$-modules:
\begin{equation}\label{s}
0\longrightarrow C_c^\infty\big((Y_E\times X_E)\setminus(Y_F\times X_F);R\big)\longrightarrow C_c^\infty(Y_E\times X_E;R)\longrightarrow C_c^\infty(Y_F\times X_F;R)\longrightarrow 0.
\end{equation}
Since the action of $\Gamma$ on the space $(Y_E\times X_E)\setminus(Y_F\times X_F)$ is free, there exists a fundamental domain $\mathcal{U}$ such that 
$$
(Y_E\times X_E)\setminus(Y_F\times X_F) = 
\bigsqcup_{i=0}^{l-1}\gamma^i\mathcal{U},
$$
and this implies that
$$
H^1\big(\Gamma,C_c^\infty\big((Y_E\times X_E)\setminus(Y_F\times X_F);R\big)\big) = 0.	
$$
Then, it follows from the long exact sequence of non-abelian cohomology corresponding to the short exact sequence (\ref{s}) that
$$
0\longrightarrow C_c^\infty\big((Y_E\times X_E)\setminus(Y_F\times X_F);R\big)^\Gamma\longrightarrow
C_c^\infty(Y_E\times X_E;R)^\Gamma\longrightarrow
C_c^\infty(Y_F\times X_F;R)\longrightarrow 0.
$$
For any $\varphi\in C_c^\infty\big((Y_E\times X_E)\setminus(Y_F\times X_F);R\big)^\Gamma$, we have
$$
\int_{(Y_E\times X_E)\setminus(Y_F\times X_F)}\varphi\,d\mu_E\,d\nu_E = \sum_{i=0}^{l-1}
\int_{\gamma^i\mathcal{U}}\varphi\,d\mu_E\,d\nu_E = 0,
$$ 
Therefore, the linear functionals $d\mu_E$ and $d\nu_E$ induces a ($Y_F\times G_n(F)$)-equivariant linear functional on the space $C_c^\infty(Y_F\times X_F;R)$, and we have following identity:
$$
\int_{Y_E}\int_{X_E}\varphi\,d\mu\,d\nu_E = c\int_{Y_F}\int_{X_F}\varphi\,d\mu_F\,d\nu_F,
$$ 
for some non-zero $c\in R$. Now, we will show that $c$ is a unit in $R$. Consider the reduction mod-$\mathfrak{m}_R$ of the above integral, i.e.,
\begin{equation}\label{int_red}
\int_{Y_E}\int_{X_E}\overline{\varphi}\,d\mu_E\,d\nu_E = 
\bar{c} \int_{Y_F}\int_{X_F}
\overline{\varphi}\,d\mu_F\,d\nu_F,
\end{equation}
where the function $\overline{\varphi}$ is the pointwise mod-$\mathfrak{m}_R$ reduction of $\varphi$. Following \cite[Chapter 1, Section 2.8]{Vigneras-l-modular}, we get a surjection $\Psi:C_c^\infty(Y_E\times G_n(E);\overline{\mathbb{F}}_l)
\longrightarrow C_c^\infty(Y_E\times X_E;\overline{\mathbb{F}}_l)$, 
defined as
$$ \Psi(f)(x,g):=\int_{N_n(E)}f(x,ng)\,dn, $$ 
for all $f\in C_c^\infty(Y_E\times G_n(E);\overline{\mathbb{F}}_l)$. Here $dn$ is a Haar measure on $N_n(E)$. There exists a $\Gamma$-invariant compact open subgroup $I\subseteq G_n(E)$ such that $\Psi(1_{M_{r\times s}(\mathfrak{o}_E)\times I})\not=0$, where $M_{r\times s}(\mathfrak{o}_E)$ is the ring of $r\times s$ matrices with entries in $\mathfrak{o}_E$, and $1_{M_{r\times s}(\mathfrak{o}_E)\times I}$ denotes the characteristic function on $M_{r\times s}(\mathfrak{o}_E)\times I$. Therefore, the Haar measure $d\mu_E$ is non-zero on the space $C_c^\infty(Y_E\times X_E;\overline{\mathbb{F}}_l)^\Gamma$, and we have $\bar{c}\not=0$. This completes the proof.
\end{proof}
\begin{remark}\label{com_int_2}\normalfont
From now on, we choose the Haar measures $d\mu_E$ and $d\mu_F$ on $X_E$ and $X_F$ respectively, such that $c=1$. Then,
$$ 
\int_{Y_E}\int_{X_E}\varphi\,d\mu_E\,d\nu_E = \int_{Y_F}\int_{X_F}\varphi\,d\mu_F\,d\nu_F. 
$$ 
Let $e$ be the ramification index of the extension $E/F$. Then, for all $r \notin\{te:t\in\mathbb{Z}\}$, we have
$$ \int_{(X_E^r)^\Gamma}\varphi\,d\mu_F=0, $$  
and for all $r\in\{te:t\in\mathbb{Z}\}$, we have
$$ \int_{Y_E}\int_{(X_E^r)^\Gamma}\varphi\,d\mu_F\,d\nu_E =
\int_{Y_F}\int_{X_F^{\frac{r}{e}}}\varphi\,d\mu_F\,d\nu_F. $$
\end{remark}
\section{Tate cohomology of generic representations of ${\rm GL}_3$}\label{generic_case}
In this section, we prove Theorem \ref{intro_3_thm}. We also record some results which are useful in proving the main theorem.

\subsection{}
Let $F$ be a finite extension of $\mathbb{Q}_p$, and let $E$ be a finite Galois extension of $F$ with $[E:F]=l$, where $l$ and $p$ are distinct odd primes. Let $\Gamma$ be the Galois group ${\rm Gal}(E/F)$. We fix a generator $\gamma$ of $\Gamma$. Fix a non-trivial additive character $\psi_F:F\rightarrow (\mathbb{Z}_l^{un})^\times$, where $\mathbb{Z}_l^{un}$ is the ring of integers of the maximal unramified extension of $\mathbb{Q}_l$ in $\overline{\mathbb{Q}}_l$. Let $\psi_E$ be the composition $\psi_F\circ{\rm Tr}_{E/F}$, where ${\rm Tr}_{E/F}$ is the trace function. Let $\Theta_F$ (resp. $\Theta_E$) be the non-degenerate character of $N_3(F)$ (resp. $N_3(E)$), defined by (\ref{non_deg_char}). The mod-$l$ reduction of $\psi_K$ (resp. $\Theta_K$) is denoted by $\overline{\psi}_K$ (resp. $\overline{\Theta}_K$).
\subsection{}
Let $(\pi,V)$ be an integral, $l$-adic, generic representation of $G_3(E)$ such that $\pi\simeq\pi^\gamma$. Let $T$ be an isomorphism from $(\pi,V)$ onto $(\pi^\gamma,V)$ with $T^l={\rm id}$. The space $V$ is endowed with an action of $\Gamma$ via the isomorphism $T$, and it is compatible with the action of $\Gamma$ on $G_3(E)$ (\cite[Proposition 6.1]{Venkatesh-Treumann}). Let $J_l(\pi)$ be the unique generic sub--quotient of the mod-$l$ reduction of $\pi$. By \cite[Lemma 2.4]{nadimpalli2024tate}, the Whittaker model $\mathbb{W}(J_l(\pi),\overline{\psi}_E)$ of $J_l(\pi)$ is also invariant under the action of $\Gamma$, defined as
$$ (\gamma.W)(g) = W(\gamma^{-1}(g)), $$
for all $g\in G_3(E)$ and $W\in \mathbb{W}(J_l(\pi),\overline{\psi}_E)$. The following is an amalgamation of \cite[Proposition 5.5, Lemma 6.2, and Proposition 6.3]{nadimpalli2024tate}. This is crucial for the main theorem.
\begin{proposition}\label{Tate_finite_generic}
Let $\Pi$ be an $l$-modular generic representation of ${\rm GL}_3(E)$ with an isomorphism $\Pi\simeq \Pi^\gamma$. Then the Tate cohomology $\widehat{H}^0(\Pi)$ has finite length as a representation of ${\rm GL}_3(F)$, and it admits a unique generic sub--quotient. Moreover, for any parabolic subgroup $P=MU$ of ${\rm GL}_3$ with Levi subgroup $M$ and unipotent radical $U$, we have the $M(F)$ equivariant isomorphism
$$ \widehat{H}^0(\Pi_{U(E)}) \simeq \widehat{H}^0(\Pi)_{U(F)}. $$
\end{proposition}

We now prove the main theorem.
\begin{theorem}\label{generic_thm}
Let $F$ be a finite extension of $\mathbb{Q}_p$, and let $E$ be a finite Galois extension of $F$ of degree $l$, where $p$ and $l$ are distinct odd primes. Let $\pi_F$ be an integral, $l$-adic generic representation of $G_3(F)$, and let $\pi_E$ be the base change lifting of $\pi_F$ to $G_3(E)$. Then, the Frobenius twist $J_l(\pi_F)^{(l)}$ occurs as a unique generic sub--quotient of the Tate cohomology group $\widehat{H}^0(J_l(\pi_E))$.
\end{theorem}
\begin{proof}
We use the local converse theorem over Artin local $\overline{\mathbb{F}}_l$-algebras (Theorem \ref{converse_thm}) to prove the theorem. Using Proposition \ref{Tate_finite_generic}, the Tate cohomology group $\widehat{H}^0(J_l(\pi_E))$ admits a unique generic sub--quotient, say $\Pi$. We aim to show that
\begin{equation}\label{claim}
\gamma_R(X,\Pi,\eta_F,\overline{\psi}_F^l) =
\gamma_R(X,J_l(\pi_F)^{(l)},\eta_F,\overline{\psi}_F^l),
\end{equation}
for all characters $\eta_F:F^\times\rightarrow R^\times$ and for all Artin local $\overline{\mathbb{F}}_l$-algebras $R$.
We divide the proof into three parts. In the first part, we use Rankin--Selberg functional equation to show that the assertion (\ref{claim}) is equivalent to the following identity of $\gamma$-factors:
$$ \gamma_R(X,\Pi,\eta_0^l,\overline{\psi}_F^l) =
\gamma_R(X,J_l(\pi_F)^{(l)},\eta_0^l,\overline{\psi}_F^l), $$
where $\eta_0$ is an $R$-valued character of $F^\times$--which is trivial on the $l$-torsion part of the finite group $k_F^\times$ and $\eta_0^l$ coincides with $\eta_F$ on the prime-to-$l$-part of $F^\times$. In the second part, we prove that
$$ \gamma_R(X^l,\Pi,\eta_0^l,\overline{\psi}_F^l) =
\gamma_R(X,J_l(\pi_E),\eta_0\circ {\rm Nr}_{E/F},\overline{\psi}_F\circ {\rm Tr}_{E/F}), $$
where ${\rm Nr}_{E/F}$ and ${\rm Tr}_{E/F}$ are the norm function and trace function respectively. The third part proves the following identity:
$$ \gamma_R(X,J_l(\pi_E),\eta_0\circ {\rm Nr}_{E/F},\overline{\psi}_F\circ {\rm Tr}_{E/F}) =
\gamma_R(X^l,J_l(\pi_F)^{(l)},\eta_0^l,\overline{\psi}_F^l). $$
\subsubsection{}
Let $R$ be an arbitrary Artin local $\overline{\mathbb{F}}_l$-algebra, and let $\eta_F$ be an $R$-valued character of $F^\times$. For all $W\in\mathbb{W}(\Pi,\overline{\psi}_F^l)$ and $U\in\mathbb{W}(J_l(\pi_F)^{(l)},\overline{\psi}_F^l)$, we have the functional equations
\begin{equation}\label{fe_1}
\sum_{m\in\mathbb{Z}}c_m^F(\widetilde{W},\eta_F^{-1};1)
q_F^mX^{-m}=\gamma_R(X,\Pi,\eta_F,
\overline{\psi}_F^l)
\sum_{m\in\mathbb{Z}}c_m^F(W,\eta_F;0)X^m,
\end{equation}
\begin{equation}\label{fe_2}
\sum_{m\in\mathbb{Z}}c_m^F(\widetilde{U},\eta_F^{-1};1)	q_F^mX^{-m} = \gamma_R(X,J_l(\pi_F)^{(l)},\eta_F,\overline{\psi}_F^l)
\sum_{m\in\mathbb{Z}}c_m^F(U,\eta_F;0)X^m.
\end{equation}
Let $k_F^\times[l]$ be the $l$-torsion part of $k_F^\times$ of order, say $l^a$, and let $Y$ be the subgroup of $\mathfrak{o}_F^\times$ whose pro-order is coprime to $l$. We have $\mathfrak{o}_F^\times = k_F^\times[l]\times Y$. Fix a generator $x_0$ of the cyclic group $k_F^\times[l]$. Then, $\eta_F(x_0)-1$ is nilpotent, and hence belongs to the maximal ideal $\mathfrak{m}_R\subseteq R$. Let $u\in \mathfrak{m}_R$ be such that $\eta_F(x_0) = 1+u$. Let $\eta_0$ be an $R$-valued character of $F^\times$ such that $\eta_0$ is trivial on $k_F^\times[l]$ and $\eta_0^l$ coincides with $\eta_F$ on the subgroup generated by $Y$ and $\varpi_F$. Then, 
\begin{align*}
c_m^F(\widetilde{W},\eta_F^{-1};1)
&=\int_F\int_{\varpi_F^m\mathfrak{o}_F^\times}\widetilde{W}
\begin{pmatrix}
	g & 0 & 0\\
	x & 1 & 0\\
	0 & 0 & 1
\end{pmatrix}\otimes\eta_F(g)^{-1}\,dg\,dx\\
&=c_m^F(\widetilde{W},\eta_0^{-l};1)+\sum_{i=1}^{l^a}
\sum_{k=1}^i(-1)^k
\binom{i}{k}
\int_F\int_{\varpi_F^mx_0^iY}\widetilde{W}
\begin{pmatrix}
	g & 0 & 0\\
	x & 1 & 0\\
	0 & 0 & 1
\end{pmatrix}\otimes\eta_0(g)^{-l}u^k\,dg\,dx.
\end{align*}
and 
\begin{align*}
c_m^F(\widetilde{U},\eta_F^{-1};1)
&=\int_F\int_{\varpi_F^m\mathfrak{o}_F^\times}\widetilde{U}
\begin{pmatrix}
	g & 0 & 0\\
	x & 1 & 0\\
	0 & 0 & 1
\end{pmatrix}\otimes\eta_F(g)^{-1}\,dg\,dx\\
&=c_m^F(\widetilde{U},\eta_0^{-l};1)+\sum_{i=1}^{l^a}
\sum_{k=1}^i(-1)^k
\binom{i}{k}
\int_F\int_{\varpi_F^mx_0^iY}\widetilde{U}
\begin{pmatrix}
	g & 0 & 0\\
	x & 1 & 0\\
	0 & 0 & 1
\end{pmatrix}\otimes\eta_0(g)^{-l}u^k\,dg\,dx.
\end{align*}
Let $\varphi$ be the characteristic function of $U_F^1$. Since both the spaces $\mathbb{W}(\Pi,\overline{\psi}_F^l)$ and $\mathbb{W}(J_l(\pi_F)^{(l)},\overline{\psi}_F^l)$ contains the compact induction ${\rm ind}_{N_3(F)}^{P_3(F)}(\overline{\Theta}_F^l)$ as $P_3(F)$-representations, there exists $W_0\in\mathbb{W}(\Pi,\overline{\psi}_F^l)$ and $U_0\in\mathbb{W}(J_l(\pi_F)^{(l)},\overline{\psi}_F^l)$ such that 
$$ W_0
\begin{pmatrix}
	g & 0 & 0\\
	0 & 1 & 0\\
	0 & 0 & 1
\end{pmatrix} = U_0
\begin{pmatrix}
	g & 0 & 0\\
	0 & 1 & 0\\
	0 & 0 & 1
\end{pmatrix}=\varphi(g). $$ 
Since both the restrictions ${\rm res}_{P_3(F)}(W_0)$ and ${\rm res}_{P_3(F)}(U_0)$ are compactly supported, we have the following identities for each  $1\leq i\leq l^a$:
\begin{multline*}
\sum_{m\in\mathbb{Z}}\Big(\int_F\int_{\varpi_F^mx_0^iY }
\widetilde{W_0}
\begin{pmatrix}
	g & 0 & 0\\
	x & 1 & 0\\
	0 & 0 & 1
\end{pmatrix}
\otimes\eta_0^l(g^{-1})\,dg\,dx\Big)q_F^mX^{-m}\\
= \gamma_R(X,\Pi,\eta_0^l,\overline{\psi}_F^l)
\sum_{m\in\mathbb{Z}}\Big(\int_{\varpi_F^mx_0^iY }\varphi(g)\otimes \eta_0^l(g)\,dg\Big)X^m
\end{multline*}
and
\begin{multline*}
\sum_{m\in\mathbb{Z}}\Big(\int_F\int_{\varpi_F^mx_0^iY }
\widetilde{U_0}
\begin{pmatrix}
	g & 0 & 0\\
	x & 1 & 0\\
	0 & 0 & 1
\end{pmatrix}
\otimes\eta_0^l(g^{-1})\,dg\,dx\Big)q_F^mX^{-m}\\
= \gamma_R(X,J_l(\pi_F)^{(l)},\eta_0^l,\overline{\psi}_F^l)
\sum_{m\in\mathbb{Z}}\Big(\int_{\varpi_F^mx_0^iY}
\varphi(g)\otimes\eta_0^l(g)\,dg\Big)X^m.
\end{multline*}
Using these relations, it follows from the functional equations (\ref{fe_1}) and (\ref{fe_2}) that the assertion (\ref{claim}) is equivalent to following identity
$$ \gamma_R\big(X,\Pi,\eta_0^l,
\overline{\psi}_F^l\big) =
\gamma_R(X,J_l(\pi_F)^{(l)},\eta_0^l,\overline{\psi}_F^l). $$
	
\subsubsection{}
Since $\Pi$ is a generic sub--quotient of $\widehat{H}^0(J_l(\pi_E))$, there exists a $G_3(F)$ stable subspace $\mathcal{M}(\pi_E,\psi_E)$ of the Tate cohomology space $\widehat{H}^0(\mathbb{W}(J_l(\pi_E),\overline{\psi}_E))$ with the following $G_3(F)$ equivariant surjective map:
$$ \Phi : \mathcal{M}(\pi_E,\psi_E)\longrightarrow \mathbb{W}(\Pi,\overline{\psi}_F^l). $$
Let $W$ be an arbitrary element of $\mathbb{W}(\Pi,
\overline{\psi}_F^l)$. Using functional equation for the pair $(\Pi,\eta_0^l)$, we get 
\begin{equation}\label{func_equ_trivial_1}
\sum_{m\in\mathbb{Z}}c_m^F(\widetilde{W},\eta_0^{-l};1)
q_F^mX^{-m} = \gamma_R\big(X,\Pi,\eta_0^l,
\overline{\psi}_F^l\big)
\sum_{m\in\mathbb{Z}}c_m^F(W,\eta_0^l;0)X^m,
\end{equation}
\begin{equation}\label{func_equ_trivial_2}
\sum_{m\in\mathbb{Z}}c_m^F(\widetilde{W},\eta_0^{-l};0)
q_F^mX^{-m} = \gamma_R\big(X,\Pi,\eta_0^l,
\overline{\psi}_F^l\big)
\sum_{m\in\mathbb{Z}}c_m^F(W,\eta_0^l;1)X^m.
\end{equation}
Let $V$ be an element in $\mathcal{M}(\pi_E,\psi_E)$ such that $\Phi(V)=W$. Set $\eta_E=\eta_0\circ{\rm Nr}_{E/F}$,
where ${\rm Nr}_{E/F}$ is the norm function. Now, the functional equation for the pair $(J_l(\pi_E),\eta_E)$ gives 
$$
\sum_{m\in\mathbb{Z}}c_m^E(\widetilde{V},\eta_E^{-1};1)
q_F^{fm}X^{-fm} = \gamma_R(X,J_l(\pi_E),\eta_E,\overline{\psi}_E)
\sum_{m\in\mathbb{Z}}c_m^E(V,\eta_E;0)X^{fm}.
$$
Using Remark \ref{com_int_2} and applying the map $\Phi$ to the above identity, we get
$$
\sum_{m\in\mathbb{Z}}c_m^F(\widetilde{W},\eta_0^{-l};1)
q_F^{lm}X^{-lm} = \gamma_R(X,J_l(\pi_E),\eta_E,\overline{\psi}_E)
\sum_{m\in\mathbb{Z}}c_m^F(W,\eta_0^l;0)X^{lm}.
$$
Similarly, we have
$$ \sum_{m\in\mathbb{Z}}c_m^F(\widetilde{W},\eta_0^{-l};0)
q_F^{lm}X^{-lm} = \gamma_R(X,J_l(\pi_E),\eta_E,\overline{\psi}_E)
\sum_{m\in\mathbb{Z}}c_m^F(W,\eta_0^l;1)X^{lm}.
$$
Hence, it follows from (\ref{func_equ_trivial_1}) and (\ref{func_equ_trivial_2}), and the preceding two 
relations that
$$
\gamma_R(X,J_l(\pi_E),\eta_E,\overline{\psi}_E) =
\gamma_R(X^l,\Pi,\eta_0^l,\overline{\psi}_F^l).
$$
	
\subsubsection{}
In this part, we show that
$$
\gamma_R\big(X,J_l(\pi_E),\eta_E,\overline{\psi}_E\big) =
\gamma_R\big(X^l,J_l(\pi_F)^{(l)},\eta_0^l, 
\overline{\psi}_F^l\big).
$$
For each $1\leq i\leq s$ and $1\leq j \leq r$ with $r,s\leq 3$, let $\rho_F^i$ and $\rho_E^j$ be the irreducible $l$-modular representations 
of $\mathcal{W}_F$ and $\mathcal{W}_E$, respectively, such that the representation $\bigoplus_{i=1}^s \rho_F^i$ (resp. $\bigoplus_{j=1}^r \rho_E^r$) is associated with the supercuspidal support of $J_l(\pi_E)$ (resp. $J_l(\pi_F)$) under the semisimple mod-$l$ local Langlands correspondence (Subsection \ref{mod-l_LLC}).
Since $\pi_E$ is the base change lift of $\pi_F$, we get from Remark \ref{mod_l_bc_form} the following isomorphism:
\begin{equation}\label{base_change}
{\rm res}_{\mathcal{W}_E}(\bigoplus_{i=1}^s \rho_F^i) \simeq 
\bigoplus_{j=1}^r \rho_E^r.
\end{equation} 
Now, using the identity (\ref{LLC}) in Subsection \ref{LLC_nilpoent}, we have
\begin{align*}
\gamma_R(X,J_l(\pi_E),\eta_E,\overline{\psi}_E)
&=\gamma_R\big(X,(\bigoplus_{j=1}^r\rho_E^j)\otimes\eta_E,
\overline{\psi}_E\big)\\
&=\gamma_R\big(X,{\rm ind}_{\mathcal{W}_E}^{\mathcal{W}_F}\big((\bigoplus_{j=1}^r
\rho_E^j)\otimes\eta_E\big),\overline{\psi}_F\big).
\end{align*}
In view of the isomorphism (\ref{base_change}), we get the following equality
\begin{equation}\label{gamma_induced}
\gamma_R\big(X,J_l(\pi_E),\eta_E,\overline{\psi}_E\big)=
\prod_{i=1}^s\gamma_R\big(X,(\rho_F^i\otimes\eta_0)
\otimes{\rm ind}_{\mathcal{W}_E}^{\mathcal{W}_F}(R),
\overline{\psi}_F\Big).
\end{equation}
For each $1\leq i\leq s$, consider the $R$-linear map
$$
(\rho_F^i\otimes\eta_0)\otimes_{R}{\rm ind}_{\mathcal{W}_E}^{\mathcal{W}_F}(R)\xrightarrow{\Sigma}
(\rho_F^i\otimes\eta_0)\otimes_{R}{\rm ind}_{\mathcal{W}_E}^{\mathcal{W}_F}
(\overline{\mathbb{F}}_l),
$$
$$
(v\otimes f)\longmapsto (v\otimes\overline{f}),
$$ 
where $\overline{f}$ is induced by $f$ via the quotient map $R\rightarrow\overline{\mathbb{F}}_l$.
\begin{enumerate}
	\item Take $v\in\rho_F^i\otimes\eta_0$ and $\phi
	\in{\rm ind}_{\mathcal{W}_E}^{\mathcal{W}_F}
	(\overline{\mathbb{F}}_l)$. 
	Since the composition $\overline{\mathbb{F}}_l\hookrightarrow R\rightarrow\overline{\mathbb{F}}_l$ is the identity map, the element $\phi\in{\rm ind}_{\mathcal{W}_E}^{\mathcal{W}_F}(R)$ and 
	$$
	\Sigma(v\otimes\phi)=v\otimes\phi.
	$$
	Therefore, $\Sigma$ is surjective.
	\item ${\rm dim}_R((\rho_F^i\otimes\eta_0)\otimes_{R}{\rm ind}_{\mathcal{W}_E}^{\mathcal{W}_F}(R)) = {\rm dim}_R((\rho_F^i\otimes\eta_0)\otimes_{R}{\rm ind}_{\mathcal{W}_E}^{\mathcal{W}_F}
	(\overline{\mathbb{F}}_l))$.
	\item $\Sigma$ is compatible with the action of $\mathcal{W}_F$.
\end{enumerate}
Thus, $(1)-(3)$ implies that the map $\Sigma$ is a $\mathcal{W}_F$ equivariant isomorphism. The Jordan--Holder series 
of $\overline{\mathbb{F}}_l[\Gamma]$ gives the following equality of $\gamma$-factors
$$
\gamma_R\big(X,\rho_F^i\otimes\eta_0\otimes{\rm ind}_{\mathcal{W}_E}^{\mathcal{W}_F}(R),\overline{\psi}_F\big) =
\gamma_R(X,\rho_F^i\otimes\eta_0,\overline{\psi}_F)^l.
$$
Using the above relation and the identity (\ref{LLC}), it follows from (\ref{gamma_induced}) that
$$
\gamma_R(X,J_l(\pi_E),\eta_E,\overline{\psi}_E) =
\gamma_R(X,J_l(\pi_F),\eta_0,\overline{\psi}_F)^l.
$$
Finally, Lemma \ref{lemma_frob} gives
$$
\gamma_R\big(X,J_l(\pi_E),\eta_E,\overline{\psi}_E\big) =
\gamma_R(X^l,J_l(\pi_F)^{(l)},\eta_0^l,\overline{\psi}_F^l).
$$
This completes the proof.
\end{proof}
\subsection{Base change of cuspidal representations of ${\rm GL}_3$}
We now derive an important consequence of Theorem \ref{generic_thm}. We follow the notations of the previous subsection. Let $E/F$ be a finite cyclic extension of $p$-adic fields with $[E:F]=l$, where $l$ and $p$ are distinct primes. Let $\pi_F$ be an integral $l$-adic cuspidal representation of $G_3(F)$, and let $\pi_E$ be the base change lifting of $\pi_F$ to $G_3(E)$. Then $\pi_E$ is also integral (Lemma \ref{bc_integral}). It follows from \cite[Chapter 1, Lemma 6.10]{Arthur_Clozel_BC} that $\pi_E$ is either a cuspidal representation, or it is a principal series representation of the form
$$ {\rm ind}_{B_3(E)}^{G_3(E)}(\eta\otimes\eta^\gamma\otimes \eta^{\gamma^2}), $$ 
where $B_3(E)$ is the group of upper triangular matrices in $G_3(E)$ and $\eta$ is an $l$-adic character of $E^\times$. In each case, the mod-$l$ reduction $r_l(\pi_E)$ is also irreducible. Then we have the following lemma.
\begin{lemma}\label{Tate_irr}
The Tate cohomology group $\widehat{H}^0(r_l(\pi_E))$ is an $l$-modular cuspidal representation of $G_3(F)$.
\end{lemma}
\begin{proof}
We first consider the case when $\pi_E$ is cuspidal. Then the mod-$l$ reduction $r_l(\pi_E)$ is also cuspidal. Moreover, the restriction ${\rm res}_{P_3(E)}(r_l(\pi_E))$ is also irreducible and is isomorphic to ${\rm ind}_{N_3(E)}^{P_3(E)}\overline{\Theta}_E$. Then, using Proposition \ref{TV_isom}, we get the following $P_3(F)$ isomorphism
$$ \widehat{H}^0(r_l(\pi_E)) \simeq {\rm ind}_{N_3(F)}^{P_3(F)}\overline{\Theta}_F^l. $$
This shows that the Tate cohomology $\widehat{H}^0(r_l(\pi_E))$ is irreducible as a representation of $P_3(F)$, and hence of $G_3(F)$. Let $P$ be a proper parabolic subgroup of $G_3$. Then using Proposition \ref{Tate_finite_generic}, we get the following isomorphism: 
$$ \widehat{H}^0(r_l(\pi_E)_{P(E)}) \simeq 
\widehat{H}^0(r_l(\pi_E))_{P(F)}. $$
Since $r_l(\pi_E)$ is cuspidal, the Jacquet module $r_l(\pi_E)_{P(E)}$ is trivial. From the above isomorphism, we get $\widehat{H}^0(r_l(\pi_E))_{P(F)}$ is trivial. This proves that $\widehat{H}^0(r_l(\pi_E))$ is cuspidal.

Next, we consider the case when $\pi_E$ is the principal series representation
$$ {\rm ind}_{B_3(E)}^{G_3(E)}(\eta\otimes\eta^\gamma\otimes \eta^{\gamma^2}). $$
The mod-$l$ reduction $r_l(\pi_E)$ is again a principal series representation ${\rm ind}_{B_3(E)}^{G_3(E)}(\overline{\eta}\otimes
\overline{\eta}^\gamma\otimes\overline{\eta}^{\gamma^2})$, where $\overline{\eta}$ is the mod-$l$ reduction of $\eta$. If $P$ is a proper parabolic subgroup of $G_3$, then the Jacquet module $r_l(\pi_E)_{P(E)}$ is $\Gamma$-invariant but it has no fixed point under the action of $\Gamma$. Therefore, the Tate cohomology $\widehat{H}^0(r_l(\pi_E)_{P(E)})$ is trivial.
From Proposition \ref{Tate_finite_generic}, we have 
$$ 
\widehat{H}^0(r_l(\pi(E))_{P(E)}) \simeq 
\widehat{H}^0(r_l(\pi_E))_{P(F)}. 
$$ 
Thus, the Jacquet module $\widehat{H}^0(r_l(\pi_E))_{P(F)}$ is also trivial. Since the Tate cohomology group $\widehat{H}^0(r_l(\pi_E))$ is of finite length and has a unique generic sub--quotient (Proposition \ref{Tate_finite_generic}), we get that $\widehat{H}^0(r_l(\pi_E))$ is irreducible, and hence cuspidal.
\end{proof}
\begin{theorem}\label{bc_cuspidal}
Let $E/F$ be a finite Galois extension of $p$-adic fields with $[E:F]=l$, where $l$ and $p$ are distinct odd primes. Let $\pi_F$ be an integral, $l$-adic cuspidal representation of $G_3(F)$, and let $\pi_E$ be the base change lifting of $\pi_F$. Then we have
$$ \widehat{H}^0(r_l(\pi_E)) \simeq r_l(\pi_F)^{(l)}. $$ 
\end{theorem}
\begin{proof}
The proof is immediate from Theorem \ref{generic_thm} and Lemma \ref{Tate_irr}.
\end{proof}

\textbf{Acknowledgements.} The author is grateful to Santosh Nadimaplli for suggesting the problem and the several discussions during this work. Thanks to Gilbert Moss for  helpful conversations.

\bibliographystyle{amsalpha}
\bibliography{bc_GL3_paper}		
	
\noindent\\
Department of Mathematics and Statistics,\\ Indian
Institute of Technology Kanpur, U.P. 208016, India.\\
\texttt{Email} : \texttt{sabya@iitk.ac.in},\, \texttt{mathsabya93@gmail.com}

\end{document}